\documentclass[11pt]{article}
\usepackage{latexsym}
\usepackage{amsmath,amsthm, amssymb}
\usepackage{amssymb,amsfonts}
\usepackage{color}
\usepackage{graphicx}
\usepackage{epsfig}
\usepackage{mathrsfs}
\usepackage{empheq}

\def\R{{ I\!\!R}}

\def\II{{\rm I\kern-0.5exI}}
\def\III{{\rm I\kern-0.5exI\kern-0.5exI}}
\numberwithin{equation}{section}
\newtheorem{theorem}{Theorem}[section]
\newtheorem{remark}[theorem]{Remark}
\newtheorem{lemma}[theorem]{Lemma}

\begin{document}
\title{ Erratum: Degenerate diffusion with a drift potential: \\
A viscosity solution approach}
\author{ Inwon C.  Kim \thanks{Department of Mathematics, UCLA.  The author is partially supported by NSF 0970072}}
\maketitle
\begin{abstract}
The earlier paper \cite{lk} contains a lower bound of the solution in terms of its $L^1$ norm, which is incorrect. In this note we explain the mistake and present a correction to it under the restriction that the permeability constant $m$ satisfies $1< m <2$.  As a consequence, the quantitative estimates on the convergence rate (Main Theorem (c) and Theorem 3.6 in \cite{lk} ) only hold for $1<m<2$. For $m\geq 2$ a partial convergence rate is obtained.
\end{abstract}

In \cite{lk},  the construction of the barrier function in step 2. of Lemma 3.4 is incorrect: this is due to the fact that the equation $u_t = (m-1)u\Delta u +|Du|^2 -C$ with $C>0$  is not well-posed when the solution becomes negative.  In the case  $1<m<2$ we present a corrected and simplified proof of Lemma 3.4, where the aforementioned error is fixed by considering an alternative equation \eqref{equation2} in the density form.  The validity of Lemma 3.4 in the case $m\geq 2$ remains open. 

Secondly, we point out that the proof and the statement of Lemma 3.5 have been originally presented just in the case $m=2$ without clarification: below we will state the general result as well as the difference in the proof.

Consequently, the results of Main Theorem (c) and Theorem 3.6 in \cite{lk} are only valid for $1<m<2$. For $m\geq 2$, the rate can be only obtained in terms of how far the free boundary of the solution is from the support of the equilibrium (see Theorem~\ref{thm:correction}).

\begin{lemma}[Lemma 3.4 in \cite{lk}, corrected version]\label{lemma1} Let $1<m<2$ and $(x_0,t_0)\in\R^n\times (0,\infty)$. Then there exists small constants $k, k', a_0>0$, depending on $m$,$n$ and the $C^2$-norm of $\Phi$ in $B_1(x_0)$, such that the following is true:
Suppose, for $0<a<a_0$,
$$
a^{-n}\int_{B_a(x_0)} \rho(\cdot,t_0) dx \geq a^k.
$$
Then $u(\cdot,t_0+a) \geq a^{k'}$ in $B_a(x_0)$. 
\end{lemma}

\begin{proof}
1. Let us define 
$$
 \tilde{u}(x,t):= u(a(x-x_0), a^2(t-t_0)).
$$
Since $u\leq 1$ in $\R^n\times[0,\infty)$,  $\tilde{u}$ satisfies, in the viscosity sense,
$$
\begin{array}{lll}
\tilde{u}_t &\geq& (m-1)\tilde{u}\Delta\tilde{u} +|D\tilde{u}|^2 -C_1a(|D\tilde{u}|+a\tilde{u})\\ \\
 & \geq & (m-1)\tilde{u}\Delta\tilde{u} + (1-C_1a)|D\tilde{u}|^2 -2C_1a,
\end{array}
$$
where the second inequality holds due to Cauchy-Schwarz inequality. Here the constant $C_1$ depends on the $C^2$-norm of $\Phi$ in $B_a(x_0)$.
Hence  $\bar{u}:=(1-C_1a)\tilde{u}$ satisfies
\begin{equation}\label{equation01}
\bar{u}_t \geq (\tilde{m}-1)\bar{u}\Delta\bar{u} + |D\bar{u}|^2 - 2C_1a,
\end{equation}
where $\tilde{m} = (1-C_1a)^{-1}(m-1)+1>m$. Choose $a_0$ small enough so that $\tilde{m}<2$. 

Therefore the corresponding density function, i.e. $\bar{\rho} = (\frac{\tilde{m}-1}{\tilde{m}}\bar{u})^{\frac{1}{\tilde{m}-1}}$ satisfies 
\begin{equation}\label{equation1}
\bar{\rho}_t \geq \Delta(\bar{\rho}^{\tilde{m}}) -\frac{2C_1}{\tilde{m}-1}a\bar{\rho}^{2-\tilde{m}} \geq \Delta (\bar{\rho}^{\tilde{m}}) -C_2a \chi_{\{\rho\geq 0\}}.
\end{equation}

2. Let $w(x,t)$ denote the weak solution of 
\begin{equation}\label{equation2}
w_t = \Delta (w|w|^{m-1} ) - C_2a\chi_{|x-x_0| \leq 2}
\end{equation}
with initial data 
$$
w(x,0)=\bar{\rho}(x,0)\chi_{|x-x_0|\leq 1}.
$$
 The weak solution $w(x,t)$ then exists  in $\R^n\times [0,\infty)$ by Theorem 5.7 of \cite{v}. 
Moreover due to \cite{dibgv}, $w$ is uniformly H\"{o}lder continuous in  $B_2(x_0)\times [1/4,1/2]$.   

Note that any nonnegative solution of the (PME), $\rho_t = \Delta(\rho^{\tilde{m}})$, is a supersolution of 
\eqref{equation2}. Therefore using an appropriate Barenblatt profile as a supersolution of \eqref{equation2} and using the fact that $\bar{\rho}(\cdot,0)\leq \chi_{|x-x_0|\leq 1}$, we have
\begin{equation}\label{domain}
\{x:w(x,t)>0\} \subset \{|x|\leq 2\}\hbox{ for } 0\leq t\leq 1/2.
\end{equation}
Therefore it follows that $w$ is a subsolution of \eqref{equation1}, and thus $w\leq \bar{\rho}$ for $t\in [0,1/2]$.

 Using \eqref{equation2} and the definition of weak solution (or formally integration by parts) yields that 
$$
\int w(x,t) dx = \int w(x,0) dx - c_nC_2at  \geq \frac{a^k}{2}-C_3at,
$$
where $c_n$ equals the volume of the $n$-dimensional ball with radius $2$. Since $k<1$, for small $a$ we have $\int w(x,1/2) dx \geq a^k/4$.  Let $x^*$ be the point where $w(\cdot,1/2)$ assumes its maximum, then from 
\eqref{domain} it follows that $|x-x^*| \leq 2$ and $w(x^*, 1/2) \geq C_4 a^k$ for some dimensional constant $C_4$. Due to the H\"{o}lder regularity of $w(\cdot,1/2)$, there exists $0<\gamma<1$ depending only on $m$ and $n$ such that 
\begin{equation}\label{mid-time}
\bar{\rho}(\cdot,1/2) \geq w(\cdot,1/2)\geq \frac{C_4}{2} a^k \hbox{ in } B_{a^{k_2}}(x^*), k_2 = \frac{k}{\gamma}.
\end{equation}

\vspace{10pt}

3. Let now $U(x,t) := B(x,t;1,C) = \frac{(C(t+1)^{2\lambda}-\frac{\lambda}{2}|x|^2)_+}{(t+1)}$ be the Barenblatt profile given in Lemma 2.18 of \cite{lk}, with $0<\lambda=((m-1)d+2)^{-1}<1/2$. Let us fix $C=a^{\lambda/2}$ such that
$$
C = a^{\lambda/2} \hbox{ (initial height)  and } \sqrt{\frac{2C}{\lambda}} = \sqrt{\frac{2}{\lambda}}a^{\lambda/4} \hbox{ (initial support size)}.
$$
If $k$ is sufficiently small, then $U(x-x^*,0) \leq \tilde{u}(\cdot,1/2)$ due to \eqref{mid-time}. Moreover, a straightforward computation yields that  $aU(\cdot,t),|DU|(\cdot,t)\leq c(t):= \sqrt{C}(t+1)^{\lambda-1}$ for $0\leq t\leq a^{-1}$.
Now let
$$
\tilde{U}(x,t): = (U(x-x^*,t)-2C_1a\int_{0}^t c(s)ds)_+
$$
Then, since $U(\cdot,t)$ is concave, we obtain
$$
\tilde{U}_t \leq (m-1)\tilde{U}\Delta\tilde{U} + |D\tilde{U}|^2 - C_1a(|D\tilde{U}|+a\tilde{U})\hbox { in } \overline{\{\tilde{U}>0\}}.
$$
Hence, by the comparison principle, $\tilde{u}(x,t+1/2)\geq \tilde{U}(x,t)$ in $ \R^n\times [0,\infty)$. In particular
$$
\tilde{u}(\cdot,a^{-1}) \geq \tilde{U} (\cdot,a^{-1}-1/2) \geq a^{1-\lambda} \hbox{ in } B_1(x_0)\subset B_3(x^*).
$$
We now conclude by scaling back to the original variable.
\end{proof}

\begin{lemma}[ Lemma 3.5 in \cite{lk}, corrected version]\label{lemma2}
Let $\mathcal{K}$ be a compact subset of $\R^n$ with $u=0$ outside of $\mathcal{K}$ for all time. Then there exists a constant $C>0$ depending on 
$m>1$, $\sup \rho$ and $\max_{x\in \mathcal{K}} \Delta\Phi(x)$ such that the following holds: Suppose
$$
\int_{B_C(0)} \rho(\cdot,t) dx \leq c_0 \hbox{ for } t_1 \leq t\leq t_2:= t_1+\log(1/c_0).
$$
Then $\rho(\cdot,t_2) \leq Cc_0^{k}$ in $B_1(0)$ with $k=\frac{2}{m(n+1)}$.

\end{lemma}
\begin{remark}
1. In step 2. of the original proof, where  we let $\tilde{\rho}=\tilde{\rho}_1+\tilde{\rho}_2$,  the initial data should be divided as follows: $\tilde{\rho}_1(\cdot,0)=\frac{\rho_0}{a}$ and $\tilde{\rho}_2(\cdot,0)=1/10$. The rest of the proof is the same.\\

2. The proof in \cite{lk} is written in the case $m=2$ without clarification: for $m\neq 2$ one has to replace the scaling for $\tilde{\rho}$ in the proof of step 2. by $\tilde{\rho}(x,t):= a^{-1}\rho(a^{m/2}x, at)$. Proceeding as before with this scaling yields the above statement. We note that in the original statement $k=1/(n+1)$.

\end{remark}

Using Lemmas~\ref{lemma1} -~\ref{lemma2}  and proceeding as in the proof of Theorem 3.6 in \cite{lk},  we obtain the following.
\begin{theorem}[Theorem 3.6 in \cite{lk}, corrected]\label{thm:correction}
Let $\Phi$ and $u_{\infty}$ be as in Theorem 3.2. Then there exists $K$ and $\alpha>0$ depending on  $m, \sup u_0, k_0, M_1$, $A:=min_{\Phi(x)>C_0} |D\Phi|$ and $n$ such that the following is true:
\begin{itemize}
\item[(a)] $\Gamma_t(u)=\partial\{u(\cdot,t)>0\}$ is in the $Ke^{-\alpha t}$-neighborhood of
the positive set $\{ u_{\infty}>0\}$.\\ \\
\item[(b)] If $1<m<2$, then $\Gamma_t(u)$ is in the $Ke^{-\alpha t} $-neighborhood of $\Gamma(u_{\infty})$. 
\end{itemize}
\end{theorem}


\begin{thebibliography}{[MT]}
\bibitem[DiBGV]{dibgv}
E. DiBenedetto, U. Gianazza and V. Vespri, {\em{Harnack estimates
for quasi--linear degenerate parabolic differential equations,}}
Acta Math., {\bf 200} (2008), 181--209.

\bibitem[KL]{lk} Inwon C. Kim and Helen K. Lei, {\em Degenerate diffusion with a drift potential: a viscosity solution approach,} DCDS-A {\bf 27} (2010), no.2, pp. 767-786.
\bibitem[V]{v} J. L.  Vazquez, ``The Porous Medium Equation: Mathematical Theory,"
Oxford Mathematical Monographs. The Clarendon Press, Oxford
University Press, Oxford, 2007.


\end{thebibliography}
\end{document}